\documentclass[a4paper,12pt]{article}
\usepackage{amscd}
\usepackage{amsmath,amsfonts,amssymb,amscd}
\usepackage{indentfirst,graphicx,epsfig}
\usepackage{graphicx,psfrag}
\input{epsf}

\title{Randi\'c Incidence Energy of Graphs\footnote{Supported by NSFC and
PCSIRT.} }
\author{\small{Ran Gu, Fei Huang, Xueliang Li }\\
{\small  Center for Combinatorics and LPMC-TJKLC}\\
{\small Nankai University, Tianjin 300071, P.R. China}\\
{\small Email: guran323@163.com, huangfei06@126.com, lxl@nankai.edu.cn,}}
\date{}

\begin{document}

\makeatletter
  \newcommand\figcaption{\def\@captype{figure}\caption}
  \newcommand\tabcaption{\def\@captype{table}\caption}
\makeatother
\newtheorem{pro}{Proposition}[section]
\newtheorem{defn}{Definition}[section]
\newtheorem{theorem}{Theorem}[section]
\newtheorem{lemma}[theorem]{Lemma}
\newtheorem{coro}[theorem]{Corollary}
\newtheorem{theo}[theorem]{Theorem}
\newenvironment{proof}{\noindent {\bf
Proof.}}{\rule{3mm}{3mm}\par\medskip}

\maketitle

\begin{abstract}
Let $G$ be a simple graph with vertex set $V(G) = \{v_1, v_2,\ldots
, v_n\}$ and edge set $E(G) = \{e_1, e_2,\ldots , e_m\}$. Similar to
the Randi\'c matrix, here we introduce the Randi\'c incidence matrix
of a graph $G$, denoted by $I_R(G)$, which is defined as the
$n\times m$ matrix whose $(i, j)$-entry is $(d_i)^{-\frac{1}{2}}$ if
$v_i$ is incident to $e_j$ and $0$ otherwise.  Naturally, the
Randi\'c incidence energy $I_RE$ of $G$ is the sum of the singular
values of $I_R(G)$. We establish lower and upper bounds for the
Randi\'c incidence energy. Graphs for which these bounds are best
possible are characterized. Moreover, we investigate the relation
between the Randi\'c incidence energy of a graph and that of its
subgraphs. Also we give a sharp upper bound for the Randi\'c
incidence energy of a bipartite graph and determine the trees with
the maximum Randi\'c incidence energy among all $n$-vertex trees. As
a result, some results are very different from those for incidence
energy.
\\[2mm]



\end{abstract}

\section{Introduction}

In this paper we are concerned with simple finite graphs. Undefined
notation and terminology can be found in \cite{Bondy}. Let $G$ be a
simple graph with vertex set $V(G) = \{v_1, v_2,\ldots, v_n\}$ and
edge set $E(G) = \{e_1, e_2,\ldots, e_m\}$, and let $d_i$ be the
degree of vertex $v_i$, $i = 1, 2, \cdots, n$.

For $S\subseteq V(G)$, $G[S]$ is used to denote the subgraph of $G$
induced by $S$. For a subset $E'$ of $E(G)$, the subgraph of $G$
obtained by deleting the edges of $E' $ is denoted by $G - E'$.  If
$E'$ consists of only one edge $e$, then $G - E'$  will be written
as $G- e$.

The Randi\'c index \cite{Randic} of $G$ is defined as the sum of
$\frac{1}{\sqrt{d_id_j}}$ over all edges $v_iv_j$ of $G$. Let $A(G)$
be the $(0,1)$-adjacency matrix of $G$ and $D(G)$ be the diagonal
matrix of vertex degrees. The Randi\'c matrix \cite{Gutman} $R=R(G)$
of order $n$ can be viewed as a weighted adjacency matrix, whose
$(i,j)$-entry is defined as
\begin{equation*}
R_{i,j}=
\left\{
  \begin{array}{ll}
   0 & \hbox{ if $i$ = $j$,} \\
    (d_id_j)^{-\frac{1}{2}} & \hbox{ if the vertices $v_i$ and $v_j$ of $G$ are adjacent}, \\
    0 & \hbox{ if the vertices $v_i$ and $v_j$ of $G$ are not adjacent.}
  \end{array}
\right.
\end{equation*}
Denote the eigenvalues of the Randi\'c matrix $R=R(G)$ by
$\rho_1,\rho_2,\ldots ,\rho_n$ and label them in non-increasing
order. The greatest Randi\'{c} eigenvalue has been studied in
\cite{Gutman}, that is, $\rho_1=1$ if $G$ possesses at least one
edge. Note that in \cite{GHL} we introduced the concepts of general
Randi\'c matrix and general Randi\'c energy and deduced some results
about them.

The signless Laplacian matrix \cite{DC} of $G$ is $Q(G) =
D(G)+A(G)$. This matrix has nonnegative eigenvalues $q_1 \geq  q_2
\geq \ldots \geq q_n \geq 0$.  If the graph $G$ does not possess
isolated vertices, the normalized signless Laplacian matrix
\cite{FRK} can be defined as
$\mathcal{Q}(G)=D(G)^{-1/2}Q(G)D(G)^{-1/2}$. Let $\mu_1^+, \mu_2^+,
\ldots ,\mu_n^+$ be  eigenvalues of $\mathcal{Q}(G)$  with
$\mu_1^+\geq \mu_2^+\geq  \ldots\geq\mu_n^+$. Then, evidently,
$$\mathcal{Q}(G)=I_n+R(G).$$
Here and later $I_n$ is denoted the unit matrix of order $n$. So
$\mu_i^+=1+\rho_i$ for $i=1,2,\ldots,n.$ Therefore, $\mu_1^+=2$.

The incidence matrix $I(G)$ of $G$ is the $n\times m$ matrix whose
$(i, j)$-entry is $1$ if  $v_i$ is incident to $e_j$ and $0$
otherwise.

The notion of the energy of a graph $G$ was introduced by Gutman
\cite{G} in 1978 as the sum of the absolute values of the
eigenvalues of $A(G)$. Its origin was from chemistry, where it is
connected with the total $\pi$-electron energy of a molecule
\cite{GP}. Research on graph energy is nowadays very active, various
properties of graph energy may be found in \cite{LSG}. The concept
of graph energy was extended to any matrix by Nikiforov \cite{VN} in
the following manner. Recall that  singular values of a real (not
necessarily square) matrix $M$ are the square roots of the
eigenvalues of the (square) matrix $MM^T$ or $M^TM$ and that these
matrices have the same nonzero eigenvalues. The energy $E(M)$ of the
matrix $M$ is then defined \cite{VN} as the sum of its singular
values.

Motivated by Nikiforov's idea, the incidence energy $IE(G)$ of a
graph $G$ was defined \cite{MJ} as the sum of the singular values of
the incidence matrix $I(G)$, that, in turn, are equal to the square
roots of the eigenvalues of $I(G)I(G)^T$.

We use $Line(G)$ to denote the line graph of $G$. It is well-known
\cite{VN} that for a graph $G$, $$I(G)I(G)^T = D(G)+A(G)=Q(G),$$ and
$$I(G)^TI(G)=2 I_m + A(Line(G)).$$
Some basic properties of incidence energy were established in
\cite{Gutman1,MJ}. Many lower and upper bounds on this quantity were
found; for details see \cite{BG,DG,GKMZ}.

Similar to the Randi\'c matrix, in this paper we define an $n\times
m$ matrix whose $(i, j)$-entry is $(d_i)^{-\frac{1}{2}}$ if $v_i$ is
incident to $e_j$ and $0$ otherwise, and call it the {\it Randi\'c
incidence matrix} of $G$ and denote it by $I_R(G)$. Obviously,
$I_R(G)=D^{-\frac 1 2}I(G)$.

Let $U$ be the set of isolated vertices of $G$, and $W=V-U$. Set
$r=|W|$ $(r\leq n)$. From the definition, we can easily get that
\begin{equation}\label{0}
I_R(G)I_R(G)^T =\left(
                 \begin{array}{cc}
                  I_{r}+ R(G[W]) & 0 \\
                   0 & 0\\
                 \end{array}
               \right)
               =\left(
                 \begin{array}{cc}
                   \mathcal{Q}(G[W]) & 0 \\
                   0 & 0\\
                 \end{array}
               \right).
\end{equation}
We can consider the Randi\'c incidence matrix as a weighted
incidence matrix. It is a natural generalization of the incidence
matrix. Let $\sigma_1(G), \sigma_2(G),\ldots,\sigma_n(G)$ be the
singular values of the Randi\'c incidence matrix of a graph $G$. Now
we define $I_RE(G) :=\sum_{ i=1}^n \sigma_i(G)$, which is called the
{\it Randi\'c incidence energy} of $G$.

Also if the graph $G$ has components  $G_1,\ldots , G_k$, such that
each of those is not an isolated vertex,  then $I_RE(G) =\sum_{
i=1}^kI_RE(G_i)$. From (\ref{0}), we know that
\begin{equation}\label{2}
I_RE(G) =\sum_{i=1}^{r} \sqrt{\mu_i^+(G[W])},
\end{equation}
and \begin{equation}\label{3'} \sum_{i=1}^n
\sigma_i^2(G)=tr(I_R(G)I_R(G)^T)=r.
\end{equation}
In particular, if $G$ has no isolated vertices,
we have that

\begin{equation} \label{1}
I_R(G)I_R(G)^T=\mathcal{Q}(G),
\end{equation}
and \begin{equation}\label{3}
\sum_{i=1}^n \sigma_i^2(G)=n.
\end{equation}

On the other hand, let us consider the $m\times m$ matrix
$I_R(G)^TI_R(G)$. It is easy to see that its $(i,j)$-entry is as
follows.
\begin{equation*}
[I_R(G)^TI_R(G)]_{i,j}=
\left\{
  \begin{array}{ll}
  \frac{1}{d_k} +\frac{1}{d_{\ell}} & \hbox{ if $i = j$, $e_i=\{v_k,v_{\ell}\}$,} \\
    \frac{1}{d_k} & \hbox{ if  $i\neq j$, $e_i$ and $e_j$ have a common vertex $v_k$ in $G$,} \\
    0 & \hbox{ if  $i\neq j$, $e_i$ and $e_j$ do not have a common vertex in $G$.}
  \end{array}
\right.
\end{equation*}
Although it looks quite different from $I^T(G)I(G)$, it has many
similar properties. For examples, the sum of each column,
respectively, each row, of the matrix $I_R(G)^TI_R(G)$ is 2, and
$tr(I_R(G)^TI_R(G))=n$. Therefore, $2$ is its an eigenvalue with
eigenvector $(1,1,\ldots, 1)^T$. Particularly, if $G$ is a
$d$-regular graph, then we have
$$I_R(G)^TI_R(G)=\frac{2}{d}I_m+\frac{1}{d}A(Line(G)).$$  Let
$\lambda_1(Line(G)),\cdots, \lambda_m(Line(G))$ denote the adjacency
eigenvalues of $Line(G)$, clearly, $$I_RE(G)=\sum_{i=1}^n
\sigma_i(G)=\sum\limits_{i = 1}^m
{\sqrt{\frac{2}{d}+\frac{1}{d}{\lambda _i}(Line(G))}}.$$

Now we give an example to calculate the Randi\'c incidence matrix
and Randi\'c incidence energy of a special class of graphs.

{\bf Example:} Consider the star $S_n$ with vertex  set
$V=\{v_1,v_2,\ldots,v_n\}$ and edge set
$E=\{e_1,e_2,\ldots,e_{n-1}\}$ where $e_i=v_iv_n$ for $i=1,2,\ldots,
n-1.$ Then  $$I_R(S_n)=\left(\begin{array}{c}
                  I_{n-1} \\
                  \alpha
                \end{array}\right),
$$ where $\alpha=(\frac{1}{\sqrt{n-1}},\frac{1}{\sqrt{n-1}},\ldots,
\frac{1}{\sqrt{n-1}}).$

$$I_R(S_n)I_R(S_n)^T=\left(
  \begin{array}{cc}
    I_{n-1} & \alpha^T \\
     \alpha & 1 \\
  \end{array}
\right)
=I_n+\left(
  \begin{array}{cc}
    O_{n-1} & \alpha^T \\
     \alpha & 0 \\
  \end{array}
\right)
=I_n+R(S_n).
$$
By calculations, we have $\mu_1^+(S_n)=2$,
$\mu_2^+(S_n)=\ldots=\mu_{n-1}^+(S_n)=1$ and $\mu_n^+(S_n)=0,$ and
then $I_RE(S_n)=n-2+\sqrt{2}.$

In this paper, we establish some lower and upper bounds for the
Randi\'c incidence energy of a graph. Graphs for which these bounds
are best possible are characterized.  Moreover, we investigate the
relation between the Randi\'c incidence energy of a graph and that
of its subgraphs. That property is analogues to incidence energy.
Also we give a sharp upper bound for the Randi\'c incidence energy
of a bipartite graph and describe the trees which have the maximum
Randi\'c incidence energy among all $n$-vertex trees. It is
interesting that the extremal trees attain the maximum Randi\'c
incidence energy are quite different from the trees which have the
maximum incidence energy.

\section{Upper and lower bounds}

\begin{theo}
Let $G$ be a graph of order $n$, and contains no isolated vertices.
Then

\begin{equation} \label{4}
I_RE(G)\geq\sqrt{n},
\end{equation}
the equality holds if and only if $G\cong K_2.$

\end{theo}

\begin{proof} It is obvious that $\sum_{i=1}^n \sigma_i \geq \sqrt{\sum_{i=1}^n \sigma_i^2}$
and the equality holds if and only if at most one of the $\sigma_i$
is non-zero. From (\ref{3}), we know that $\sum_{i=1}^n
\sigma_i^2(G)=n$.  Therefore

$$I_RE(G) =\sum_{i=1}^n \sigma_i \geq \sqrt{\sum_{i=1}^n \sigma_i^2}=\sqrt{n}.$$

The equality holds if and only if at most one of the $\sigma_i$ is
non-zero, that is, the rank of $(I_R(G)I_R(G)^T)$ is 1. This is
equivalent to $rank(I_R(G))=1$ since
$rank(I_R(G)I_R(G)^T)=rank(I_R(G))$. Therefore, each component of
$G$ must have exactly one edge, i.e., each component is isomorphic
to $K_2$. If the graph $G$ has more than one component, clearly,
$rank(I_R(G))>1$, a contradiction, and hence $G\cong K_2$.
\end{proof}

\noindent{\bf Remark 2.1:} Note that in \cite{MJ} there is a similar
lower bound for incidence energy similar to (\ref{4}), that is,
$IE(G)\geq\sqrt{2m}$, the equality holds if and only if $m \leq 1$.

\begin{theo}
Let $G$ be a graph of order $n$ ($n\geq2$), and contains no isolated
vertices. Then
\begin{equation} \label{5}
 I_RE(G)\leq \sqrt{2}+\sqrt{(n-1)(n-2)},
 \end{equation}
the equality holds if and only if $G\cong K_n$.

\end{theo}

\begin{proof}
By applying the Cauchy-Schwartz inequality we have that
 $$\sum_{i=2}^n \sqrt{\mu_i^+(G)}\leq  \sqrt{(n-1)\sum_{i=2}^n\mu_i^+(G)},$$
with  equality holds if and only if
$\mu_2^+(G)=\mu_3^+(G)=\ldots=\mu_n^+(G)=\frac{n-2}{n-1}=
1-\frac{1}{n-1}.$ Since $\sum_{i=1}^n \mu_i^+(G)=n$ and
$\mu_1^+(G)=2$ if $G$ has at least one edge, then
 \begin{eqnarray*}
 I_RE(G)&=& \sum_{i=1}^n \sqrt{\mu_i^+(G)}= \sqrt{2}+\sum_{i=2}^n \sqrt{\mu_i^+(G)}\\
 &\leq& \sqrt{2}+\sqrt{(n-1)(n-\mu_1^+(G)}=\sqrt{2}+\sqrt{(n-1)(n-2)}.
\end{eqnarray*}

The equality is attained if and only if
$\mu_2^+(G)=\mu_3^+(G)=\ldots=\mu_n^+(G)=1-\frac{1}{n-1}$. Then,
$\rho_2(G)=\rho_3(G)=\ldots=\rho_n(G)=-\frac{1}{n-1}$, and therefore
we have $G\cong K_n.$

Conversely, if $G\cong K_n$, we can easily check that $I_RE(G)=
\sqrt{2}+\sqrt{(n-1)(n-2)}$. \end{proof}

\noindent{\bf Remark 2.2:} In \cite{MJ} there is an upper bound for
incidence energy, that is, $IE(G)\leq\sqrt{2mn}$, the equality holds
if and only if $m=0$. That result is quite different from ours.

\section{ Randi\'c incidence energy of subgraphs}

At the beginning of this section, we review some concepts in matrix
theory.

Let $A$ and $B$ be complex matrices of order $r$ and $s$,
respectively $(r \geq s)$. We say the eigenvalues of $B$ interlace
the eigenvalues of $A$, if $\lambda_i(A)\geq \lambda_i(B)\geq
\lambda_{r-s+i}(A)$ for $i =1,\ldots , s$.

\begin{lemma} \cite[p.51]{Doob}
If $A$ and $B$ are real symmetric matrices of order $n$ and $C = A
+B$, then
\begin{eqnarray*}
 \lambda_{i+j+1} (C)&\leq &  \lambda_{i+1} (A)+ \lambda_{j-1} (B)\\
  \lambda_{n-i-j} (C)&\geq &  \lambda_{n-i} (A)+ \lambda_{n-j} (B)
\end{eqnarray*}
for $i, j = 0, \ldots , n$ and $i + j\leq n - 1$. In particular, for
all integer $i$ $(1 \leq i \leq n)$,
\begin{equation}\label{6}
\lambda_i(C) \geq \lambda_i(A) + \lambda_n(B).
\end{equation}

\end{lemma}

\begin{theo}
Let $G$ be a graph and $E'$ be a nonempty subset of $E(G)$. Then
\begin{equation}\label{7}
I_RE(G)>I_RE(G-E').
\end{equation}

\end{theo}

\begin{proof}
Let $H$ be the spanning subgraph of $G$ such that $E(H)=E'$. The
Randi\'c incidence matrix of $G$ can be partitioned as
$I_R(G)=\left(
\begin{array}{cc}
I_R(H) & I_R(G-E') \\
\end{array}
 \right),$ and so $$I_R(G)I_R(G)^T=I_R(H)I_R(H)^T+I_R(G-E')I_R(G-E')^T.$$

Since $I_R(H)I_R(H)^T$ is positive semi-definite, by Eq.(\ref{6}),
$\lambda_i(I_R(G)I_R(G)^T)\geq \lambda_i(I_R(G-E')I_R(G-E')^T)$ for
$i=1,2,\ldots,n.$  It follows that $I_RE(G)\geq I_RE(G-E').$

Moreover, $\lambda_i(I_R(G)I_R(G)^T)=
\lambda_i(I_R(G-E')I_R(G-E')^T)$ for all  $i$ if the equality holds.
Consequently, $tr(I_R(G)I_R(G)^T)=tr(I_R(G-E')I_R(G-E')^T)$, and it
implies that $tr(I_R(H)I_R(H)^T)=0$. Since $I_R(H)I_R(H)^T$ is
positive semi-definite, $\lambda_i(I_R(H)I_R(H)^T)=0$,
$i=1,2,\ldots,n$. $H$ must be an empty graph, a contradiction.
\end{proof}

\noindent{\bf Remark 3.1:} For incidence energy, in \cite{MJ} there
is an analogous theorem, that is, the incidence energy of a graph is
greater than that of its proper subgraphs. \vspace{3ex}

According to $I_RE(K_n)= \sqrt{2}+\sqrt{(n-1)(n-2)}$, the following corollary is
obvious.
\begin{coro}
Let $G$ be a non-empty graph with clique number $c$. Then $I_RE(G)\geq \sqrt{2}+\sqrt{(c-1)(c-2)}$. In particular, if $G$ has at least one edge then $I_RE(G)\geq \sqrt{2}$.
\end{coro}

When the edge subset $E'$ consists of exactly one edge, we have the following theorem.
\begin{theorem}\label{th11}
Let $G$ be a connected graph, $e=\{uv\}$ be an edge of $G$. Then
\begin{equation}\label{eq3}
I_RE(G) \geq\sqrt{\frac{1}{d(u)}+\frac{1}{d(v)}+[I_RE(G- e)]^2},
\end{equation}
where $d(u)$ and $d(v)$ denote the degree of $u$ and $v$, respectively. Moreover, the equality holds if and only if $G \cong K_2$.
\end{theorem}

\begin{proof}
It is easy to see that the Randi\'c incidence matrix of $G$ can be
represented in the form of $$I_R(G)=(I_R(G- e)\quad \beta), $$ where
$\beta$ is a vector of size $n$ whose first two components are
$\frac{1}{\sqrt{d(u)}}$ and $\frac{1}{\sqrt{d(v)}}$, the other
components are 0. Therefore,
$$I_R(G)I_R(G)^T=I_R(G- e)I_R(G- e)^T+\left( {\begin{array}{*{20}{c}}
J&O\\
0&0
\end{array}} \right),$$ where $J=\left( {\begin{array}{*{20}{c}}
{\frac{1}{{d(u)}}}&{\frac{1}{{\sqrt {d(u)d(v)} }}}\\
{\frac{1}{{\sqrt {d(u)d(v)} }}}&{\frac{1}{{d(v)}}}
\end{array}} \right)$.
Hence,
\begin{equation}\label{eq4}
tr(I_R(G)I_R(G)^T) =\frac{1}{d(u)}+\frac{1}{d(v)}+tr(I_R(G- e)I_R(G- e)^T),
\end{equation}
By  (\ref{6}), we have $\sigma_i(G)\geq \sigma_i(G- e)$ for all $i=1,\cdots,n$.
\begin{eqnarray*}
[I_RE(G)]^2 &= & \sum\limits_i {{\sigma _i}^2(G)} + 2\sum\limits_{i < j} {{\sigma _i}(G){\sigma _j}(G)} \\
\null  &=&  tr(I_R(G)I_R(G)^T) +2\sum\limits_{i < j} {{\sigma _i}(G){\sigma_j}(G)}\\
\null  &=& \frac{1}{d(u)}+\frac{1}{d(v)}+tr(I_R(G- e)I_R(G- e)^T)+2\sum\limits_{i < j} {{\sigma _i}(G){\sigma_j}(G)}\qquad  ( by \; (\ref{eq4}))\\
\null  &=&\frac{1}{d(u)}+\frac{1}{d(v)}+\sum\limits_i {{\sigma _i}^2(G- e)}+2\sum\limits_{i < j} {{\sigma _i}(G){\sigma _j}(G)}\\
\null  &\geq&\frac{1}{d(u)}+\frac{1}{d(v)}+\sum\limits_i {{\sigma _i}^2(G- e)}+2\sum\limits_{i < j} {{\sigma _i}(G- e){\sigma _j}(G- e)}\\
\null  &=&\frac{1}{d(u)}+\frac{1}{d(v)}+[I_RE(G- e)]^2.
\end{eqnarray*}
If $I_R(G)$ has at least two non-zero singular values, since
$I_RE(G)>I_RE(G- e)$, then ${\sigma _k}(G)>{\sigma _k}(G- e)$ for
some $2\leq k\leq n$. Thus,
$$\sum\limits_{i < j} {{\sigma _i}(G){\sigma _j}(G)}
= {\sigma _1}(G){\sigma _k}(G) + \sum\limits_{i < j,(i,j) \ne (1,k)} {{\sigma _i}(G){\sigma _j}(G)}.$$
Since ${\sigma _1}(G){\sigma _k}(G)>0$, we have that
$$\sum\limits_{i < j} {{\sigma _i}(G){\sigma _j}(G)}>\sum\limits_{i < j} {{\sigma _i}(G- e){\sigma _j}(G- e)}.$$
Thus, if the Randi\'c incidence matrix of the graph $G$ has more
than one non-zero singular value, the equality in (\ref{eq3}) does
not occur. Since $rank(I_R(G)I_R(G)^T)= rank(I_R(G))$, then in the
equality case, $rank(I_R(G))$ must be equal to 1. On the other hand,
if the graph $G$ has more than one edge, $rank(I_R(G))>1$.
Therefore, the equality in (\ref{eq3}) holds if and only if $G =
K_2$.
\end{proof}

\noindent{\bf Remak 3.2:} Here we point out that if we set
$d(u)=d(v)=1$ in (\ref{eq3}), it is just the relation between
$IE(G)$ and $IE(G-e)$ given in \cite{MJ}.

\section{Upper bound for bipartite graphs}
\begin{theorem}\label{thbi}
Let $G$ be a bipartite graph of order $n$ without isolated vertices. Then
\begin{equation}\label{eqbi}
I_RE(G) \leq n-2+\sqrt{2},
\end{equation}
the equality holds if and only if $G$ is a complete bipartite graph.
\end{theorem}
\begin{proof}
Let $Q(G)$ be the signless Laplacian matrix of $G$, and $q_1\geq
q_2\geq\ldots\geq q_n$ be the signless Laplacian spectrum of $G$. It
is well known that if $G$ is bipartite, $q_n=0$. Since
$\mathcal{Q}(G)=D(G)^{-\frac{1}{2}}Q(G)D(G)^{-\frac{1}{2}}$, we get
that $\mu_n^+=0$. Hence, $\sum\limits_{i = 1}^{n - 1} \mu_i^+ =n$.
Since $G$ has at least one edge, $\mu_1^+=2$. By Cauchy-Schwarz
inequality, we have
$$\sum\limits_{i = 2}^{n - 1} \sqrt{\mu_i^+}\leq\sqrt{(n-2)(n-\mu_1^+})=\sqrt{(n-2)(n-2)}=n-2.$$
So, we have
$$I_RE(G) =\sqrt{\mu_1^+}+\sum\limits_{i = 2}^{n - 1} \sqrt{\mu_i^+}\leq n-2+\sqrt{2},$$
the equality holds if
$\mu_2^+=\cdots=\mu_{n-1}^+=\frac{n-2}{n-2}=1$, and $\mu_1^+=2$,
$\mu_n^+=0$. This implies that the Randi\'c eigenvalues $\rho_i$ of
$G$ satisfies $\rho_2=\cdots=\rho_{n-1}=0$, $\rho_1=1$ and
$\rho_n=-1$, i.e., $\rho_1=1$ is the only positive Randi\'c
eigenvalue of $G$. From Theorem 2.4 in \cite{Gutman}, $G$ is a
complete multipartite graph. Since $G$ is a bipartite graph, we
derive that $G$ is a complete bipartite graph. Conversely, let
$G=(X,Y)$ be a complete bipartite graph with two vertex classes $X$
and $Y$, where $|X|=x$, $|Y|=y$ and $x+y=n$. As is known in
\cite{Doob}, the adjacency eigenvalues of $G$ are $\sqrt{xy}$,
$-\sqrt{xy}$, 0 ($n-2$ times). It is easy to get that
$\mathcal{Q}(G)=I_n+\frac{1}{\sqrt{xy}}A(G)$, where $A(G)$ is the
adjacency matrix of $G$. So, we have $\mu_1^+=2$, $\mu_n^+=0$ and
$\mu_2^+=\cdots=\mu_{n-1}^+=1$. Thus, $I_RE(G)= n-2+\sqrt{2}$.
\end{proof}

Because trees are bipartite graphs, we characterized the unique tree
with maximum  Randi\'c incidence energy.
\begin{theorem}\label{corbi}
Among all trees with $n$ vertices, the star $S_n$ is the unique
graph with maximum Randi\'c incidence energy.
\end{theorem}
\begin{proof}
Let $G$ be a tree with $n$ vertices. Suppose that $G$ is not a
complete bipartite graph. Then $G$ is a spanning subgraph of some
complete bipartite graph $K_{s,t}$, where $s+t=n$. So,
$I_RE(G)<I_RE(K_{s,t})$. From Theorem \ref{thbi}, we know that
$I_RE(K_{1,n-1})=I_RE(K_{s,t})=n-2+\sqrt{2}$. So,
$I_RE(G)<I_RE(K_{1,n-1})$.

If $G$ is a complete bipartite graph $K_{x,y}$, then one of $x$ and
$y$ must be equal to 1; otherwise, there exists a cycle in $G$, a
contradiction. So $G$ must be the star $K_{1,n-1}$. Hence, the star
$K_{1,n-1}$ is the unique graph with maximum Randi\'c incidence
energy among all $n$-vertex trees.
\end{proof}

\noindent{\bf Remak 4.1:} The same problem has been studied for
incidence energy in \cite{Gutman1}, where the authors proved that
for any $n$-vertex tree $T$, $IE(S_n) \leq IE(T) \leq IE(P_n)$,
where $P_n$ denotes the path on $n$ vertices. But, our result says
that $I_RE(T) \leq I_RE(S_n)$. However, we do not know if $I_RE(P_n)
\leq I_RE(T)$ holds. For Randi\'c index, it was showed in
\cite{CGHP,Yu} that among trees with $n$ vertices, the star $S_n$
has the minimum Randi\'c index and the path $P_n$ attains the
maximum Randi\'c index.\\

To end this paper, we point out that one can generalize the concepts
Randi\'c incidence matrix and Randi\'c incidence energy. Similar to
those in \cite{GHL}, define the general Randi\'c incidence matrix of
a graph $G$ as an $n\times m$ matrix whose $(i, j)$-entry is
$(d_i)^{\alpha}$ if $v_i$ is incident to $e_j$ and $0$ otherwise,
where $\alpha\neq 0$ is a fixed real number. Also, define the
general Randi\'c incidence energy as the sum of the singular values
of the general Randi\'c incidence matrix of a graph $G$. It could be
interesting to further study these generalized concepts and get some
unexpected results.

\end{document}